\documentclass[preprint]{siamart1116}
\usepackage{mathrsfs,latexsym,amsmath,amssymb,graphicx,textcomp,multirow,stmaryrd,caption,mathtools}
\newcommand{\bx}{{\bf x}}

\def\esssup{\text{ess sup}}

\usepackage{epstopdf}
\newcommand{\vertiii}[1]{{\left\vert\kern-0.25ex\left\vert\kern-0.25ex\left\vert #1
    \right\vert\kern-0.25ex\right\vert\kern-0.25ex\right\vert}}
\newsiamthm{remark}{Remark}

\begin{document}

\title{An ensemble algorithm for numerical solutions to deterministic and random parabolic PDEs}
\author{
Yan Luo \thanks{School of Mathematics, Sichuan University, No.24 South Section 1, Yihuan Road, Chengdu, Sichuan 610064, P. R. China and School of Mathematical Sciences, University of Electronic Science and Technology of China, No.2006, Xiyuan Ave, West Hi-Tech Zone, Chengdu, Sichuan 611731, P. R. China. Research supported by the Young Scientists Fund of the National Natural Science Foundation of China grant 11501088.}
\and Zhu Wang\thanks{Department of Mathematics,
University of South Carolina, 1523 Greene Street, Columbia, SC 29208, USA (\email{wangzhu@math.sc.edu}). Research  supported by the
U.S. National Science Foundation grant DMS-1522672 and the U.S. Department of Energy grant DE-SC0016540.}
}

\maketitle

\begin{abstract}
In this paper, we develop an ensemble-based time-stepping algorithm to efficiently find numerical solutions to a group of linear, second-order parabolic partial differential equations (PDEs). Particularly, the PDE models in the group could be subject to different diffusion coefficients, initial conditions, boundary conditions, and body forces.
The proposed algorithm {  leads to a single discrete system for the group with multiple right-hand-side vectors by introducing an ensemble average of the diffusion coefficient functions and using a new semi-implicit time integration method}. 
The system could be solved more efficiently than multiple linear systems with a single right-hand-side vector.
We first apply the algorithm to deterministic parabolic PDEs and derive a rigorous error estimate that shows the scheme is first-order accurate in time and is optimally accurate in space.
We then extend it to find stochastic solutions of parabolic PDEs with random coefficients and put forth an ensemble-based Monte Carlo method.
The effectiveness of the new approach is demonstrated through theoretical analysis.
Several numerical experiments are presented to illustrate our theoretical results.
\end{abstract}



\section{Introduction}
In many application problems involving numerical simulations of PDEs, one is not interested in a single simulation, but a number of simulations with different computational settings such as distinct initial condition, boundary conditions, body force, and physical parameters.
For instance, data assimilation methods used in meteorology, such as the ensemble Kalman filter \cite{kalnay2003atmospheric}, run a numerical weather model forward for many times by perturbing initial conditions and uncertain parameters.
The ensemble of outputs is then used to update not only {the} forecast, but also the covariance matrix.
Similarly, when a random sampling method is used in {uncertainty} quantification, a model is run forward for a large number of times at the selected parameter sample values in order to collect outputs and determine the underlying statistical information \cite{gunzburger2014stochastic}.
Such a group of simulations is computationally expensive, especially, when the forward simulation is of a large scale.

Aiming at developing a fast algorithm for such applications, an
ensemble time-stepping algorithm was proposed in
\cite{jiang2014algorithm}, where a group of size $J$ Navier-Stokes
equations (NSE) with different initial conditions and forcing
terms is simulated. All solutions are found by
solving a single linear system with one shared coefficient matrix
and $J$ right-hand-side (RHS) vectors at each time step. Thus, it reduces the
storage requirements and computational cost for the group of
simulations. The algorithm is first-order accurate in time, which
was extended to higher-order accurate schemes in
\cite{jiang2015higher,jiang2017second}. For high Reynolds number
incompressible flows, ensemble regularization methods were developed in
\cite{jiang2015higher,jiang2015numerical,takhirov2015time}, and a
turbulence model based on ensemble averaging was developed in
\cite{jiang2015analysis}. The ensemble algorithm has also been
extended to simulate MHD flows in \cite{mohebujjaman2017efficient}
and to the reduced-order modeling setting in
\cite{gunzburger2017ensemble, Gunzburger2016higher}. For
parametrized flows, the ensemble algorithms were developed in
\cite{Gunzburger2016efficient} for multiple numerical simulations subject
to not only different initial condition, boundary conditions and body force,
but also physical parameters.
It is worth mentioning that the ensemble method is only applied to problems with constant physical parameters so far.

In this paper, we consider a group of numerical solutions to second-order parabolic PDEs.
We first develop an ensemble algorithm for deterministic problems, in which the physical parameter---diffusion coefficient---is a function {of} space and time; and then extend the method to stochastic problems.
To our knowledge, this is the {\em first} time that an ensemble scheme is derived for problems with non-constant parameters and is further applied to PDEs with random coefficients.

The initial boundary value problem we consider is as follows. 
\begin{equation}
\left\{
\begin{aligned}
u_t-\nabla\cdot [a(\bx,t)\nabla u] &=f(\bx,t), & \text{  in  } D\times(0,T){  ,} \\
u(\bx,t)&=g(\bx, t), \qquad  & \text{ on } \partial D\times (0,T){  ,}\\
u(\bx,0)&=u^0(\bx), & \text{  in  } D{  ,}
\end{aligned}
\right.
 \label{Eq10}
 \end{equation}
where $D$ is a bounded Lipschitz domain in $\mathbb{R}^d$ for
$d=1, 2, 3$, the diffusion coefficient $a(\bx,t)\in
L^2(W^{1,\infty}(D); 0, T)$, body force $f(\bx, t)\in
L^2(H^{-1}(D); 0, T)$, initial condition $u^0(\bx)\in H_0^1(D)\cap
H^{l+1}(D)$ with $l\geq 1$. 

We consider the setting in which one {needs} to run a group of simulations. Each of these simulations is subject to independent initial and boundary conditions, body force and diffusion coefficient.
Suppose the ensemble of simulations includes $J$ independent members, in which the $j$-th member satisfies: 
\begin{equation}
\left\{
\begin{aligned}
u_{j, t}-\nabla\cdot [a_j(\bx, t)\nabla u_j] &=f_j(\bx, t), & \text{  in  } D\times(0,T){  ,} \\
u_j(\bx, t)&=g_j(\bx, t),  & \text{ on } \partial D\times (0,T){  ,}\\
 u_j(\bx,0)&=u_{j}^0(\bx), & \text{  in  } D{  ,}
 \end{aligned}
 \right.
 \label{Eq6}
 \end{equation} 
for $j=1, 2, \ldots, J$. 
In general, when an implicit time stepping method is applied to the above system, the related discrete system would change as $j$ varies.
As a result, one {needs} to solve $J$ linear systems at each time step.
In order to improve its efficiency, we develop an ensemble-based time-stepping scheme.
For illustrating the main idea, we first present the semi-discrete (in time) system.

{  For simplicity, we consider a uniform time partition on $[0, T]$ with the step size $\Delta t$.}  
{Denote by $u_j^n$, $a_j^n$ and $f_j^n$ the functions $u_j$, $a_j$ and $f_j$ evaluated at the time instant  $t_n=n\Delta t$.}
Define the ensemble mean of the diffusion coefficient functions at time {  $t_n$} by
\begin{equation}
\overline{a}^n:=\frac{1}{J}\sum\limits_{j=1}^Ja_j(\bx, t_n).
\label{eq:abar}
\end{equation}
The new ensemble-based time stepping scheme reads, for $j=1,\ldots,J${:}
 \begin{equation}
 \frac{u_j^{n+1}-u_j^n}{{  \Delta t}}-\nabla\cdot({\overline{a}^{n+1}}\nabla
 u_j^{n+1})-\nabla\cdot[(a_j^{n+1}-{\overline{a}}^{n+1})\nabla
 u_j^n]=f_j^{n+1}
 \label{e0}
 \end{equation}
with the same boundary and initial conditions as those in \eqref{Eq6}.
Rearranging the equation gives
\begin{equation}
 \frac{1}{\Delta t}{u_j^{n+1}}-\nabla\cdot({\overline{a}}^{n+1}\nabla u_j^{n+1})
 =f_j^{n+1}+ \frac{1}{\Delta t}{u_j^n}+\nabla\cdot[(a_j^{n+1}-{\overline{a}}^{n+1})\nabla
 u_j^n]. \label{ens1}
 \end{equation}
	It is easy to see, after a spatial discretization, the coefficient matrix of the resulting linear system will be independent {of} $j$.
This represents the key feature of the ensemble method, that is, the discrete systems associated {with} an ensemble of simulations share a unique coefficient matrix and only the RHS vectors vary among the ensemble members.
Thus, when the considered problem has a small scale, one only {needs} to do LU factorization of the coefficient matrix once and use it to find the solutions of the group; 
when the problem is of a large scale, the proposed scheme together with the block Krylov subspace iterative methods will improve the efficiency of a number of numerical simulations (see, e.g. \cite{parks2016block} and references therein).

 The rest of this paper is structured as follows.
 In Section \ref{sec:notation}, we introduce some notations and mathematical preliminaries.
 In Section \ref{sec:deterministic}, we analyze the ensemble scheme (\ref{e0}) under a full discretization in both space and time, and prove its stability and convergence.
In Section \ref{sec:random}, we discuss the random parabolic equations and provide stability analysis and error estimations.
Several numerical experiments are presented in Section \ref{sec:num}, which illustrates the effectiveness of our proposed scheme on both deterministic and random parabolic problems.
A few concluding remarks are given in Section \ref{sec:con}.

 \section{Notations and preliminaries}
 \label{sec:notation}
 Denote the $L^2(D)$ norm and inner product by $\|\cdot\|$ and $(\cdot,\cdot)$, respectively.
 {Let $L^\infty(D)$ be the set of bounded measurable functions equipped with the norm 
 $|v|_{\infty}:=\esssup_{\bx\in D} |v|$.} 
 Denoted by $W^{s, q}(D)$ the Sobolev space of functions having generalized
 derivatives up to the order $s$ in the space $L^q(D)$, where $s$ is a nonnegative integer and $1\leq q\leq +\infty$.
 The equipped Sobolev norm of $v\in W^{s,q}(D)$ is denoted by $\|v\|_{W^{s,q}(D)}$.
 When $q=2$, we use the notation $H^s(D)$ instead of $W^{s, 2}(D)$.
 As usual, the function space $H_0^1(D)$
 is the subspace of $H^1(D)$ consisting  of functions that vanish
 on the boundary of $D$ in the sense of trace, equipped with the
 norm $\|v\|_{H^1_0(D)}=\left(\int_D|\nabla v|^2\, d\bx\right)^{1/2}$.
 When $s= 0$, we shall keep the notation with $L^q(D)$ instead of $W^{0, q}(D)$. 
 The space $H^{-s}(D)$ is the dual space of bounded linear
 functions on $H_0^s(D)$. A norm for $H^{-1}(D)$ is defined by
 $
 \|f\|_{-1}=\sup_{0\neq v\in H_0^1(D)}{(f,v)}/{\|\nabla
 v\|}.
 $
 

{   Stochastic functions have different structures. 
Let $(\Omega,\mathcal{F},P)$ be a complete probability space, where $\Omega$ is the set of outcomes, $\mathcal{F}\subset 2^\Omega$ is the $\sigma-$algebra of events, and $P: \mathcal{F} \rightarrow [0, 1]$ is a probability measure. If $Y$ is a random variable in the space and belongs to $L_P^1(\Omega)$, its expected value is defined by 
$$E[Y]=\int_{\Omega}Y(\omega)dP(\omega).$$}
With the multi-index notation,
 $\alpha=(\alpha_1,\ldots,\alpha_d)$ is a $d$-tuple of nonnegative
 integers with the length $|\alpha|=\sum_{i=1}^d\alpha_i$.
The stochastic Sobolev spaces
 $\widetilde{W}^{s,q}(D)=L^q_P(\Omega,W^{s,q}(D))$ containing
 stochastic function, $v:\Omega\times D\rightarrow R$, that are
 measurable with respect to the product $\sigma$-algebra $\mathcal{F}\bigotimes
 B(D)$ and equipped with the averaged norms $\|v\|_{\widetilde{W}^{s,q}(D)}= \left( E\left[ \|v\|^q_{W^{s,q}(D)} \right] \right)^{1/q}
 = \left( E \left[ \sum_{|\alpha\leq s|}\int_D|\partial^\alpha v|^qd\bx \right] \right)^{1/q},1\leq
 q<+\infty$.
 Observe that if $v\in \widetilde{W}^{s,q}(D)$, then $v(\omega,\cdot)\in
 W^{s,q}(D)$ almost surely (a.s.) and $\partial^\alpha v(\cdot, \bx)\in
 L^q_P(\Omega)$ almost everywhere (a.e.) on the $D$ for $|\alpha|\leq s$. Whenever
 $q=2$, the above space is a Hilbert space, i.e., $\widetilde{W}^{s,2}(D)=\widetilde{H}^s(D)\simeq L^2_P(\Omega)\bigotimes
 H^s(D)$. In this paper, we consider the tensor product Hilbert space $H=\widetilde{L}^2(H^1_0(D); 0,T)\simeq L^2_P(\Omega; H_0^1(D); 0,T)$
 endowed with the inner product $(v,u)_H\equiv E\left[\int_0^T\int_D\nabla v\cdot\nabla u\,d\bx\,dt\right]$.\\

\section{The ensemble scheme of deterministic parabolic equations}
\label{sec:deterministic}
We first consider the deterministic parabolic equations and analyze the full discretization of the proposed ensemble scheme \eqref{e0}. For simplicity of presentation, we consider the problem with a homogeneous boundary condition, {   while the nonhomogeneous cases can be similarly analyzed incorporating the method of shifting (see Section 5.4 in \cite{brennermathematical}). Furthermore, we include numerical test cases with nonhomogeneous boundary conditions in Section \ref{sec:num}. 
}

{ 
Suppose the following two conditions are valid: 
\begin{itemize} 
\item[(i)] There exists a positive constant $\theta$ such that, for any $t\in [0,T]$, 
\begin{equation}
\min\limits_{\bx\in {\overline{D}}}a(\bx, t) \geq \theta;
 \label{eq:cond_coercivity}
\end{equation}
\item[(ii)] There exist positive constants $\theta_{-}$ and $\theta_{+}$ such that, for any $t\in [0,T]$, 
\begin{equation}
\theta_{-}\leq |a_j(\bx, t)-\overline{a}(\bx, t)|_\infty\leq\theta_{+}.
 \label{eq:cond_bound}
\end{equation}
\end{itemize}
Obviously, Condition (i) guarantees the uniform coercivity of the problem; Condition (ii) states that the distance from coefficient $a_j(\bx, t)$ to the ensemble average $\overline{a}(\bx, t)$ is uniformly bounded. 
}

Let ${\mathcal{T}}_h$ be a quasi-uniform triangulation of the domain $D$, made of elements $K$, such that $\overline{D}=\bigcup_{K\in{\mathcal{T}}_h}\overline{K}$.
Define the mesh size $h:=\max_{K \in \mathcal{T}_h } h_K$,
where $h_K$ is the diameter of the element $K$.
Denoted by $V_h$ the finite element space
$$
 V_h:=\{v\in H_0^1(D)\cap H^{l+1}(D); v|_K \text{ is a polynomial of degree } l, \forall
K\in {\mathcal{T}}_h \}.
$$
{  With the assumed uniform time partition on $[0, T]$ and set $N= T/\Delta t$,} 
the fully discrete approximation of \eqref{e0} is as follows:
Find $u_{j,h}^{n+1}\in V_h$ satisfying, for {$n= 0, \ldots, N-1$ and $j=1, \cdots, J$ }:
\begin{equation}
\begin{aligned}
\left(\frac{u_{j,h}^{n+1}-u_{j,h}^n}{\Delta t},v_h\right)
+\left(\overline{a}^{n+1}\nabla u_{j,h}^{n+1},\nabla v_h \right)
&+\left((a_j^{n+1}-\overline{a}^{n+1})\nabla u_{j,h}^n,\nabla v_h \right) \\
&=
\left(f_{j}^{n+1}, v_h\right),  \quad \forall\, v_h\in V_h
\end{aligned}
\label{S0}
\end{equation}
with the initial condition {$u_{j,h}^0\in V_h$ satisfying $(u_{j, h}^0, v_h) = (u_j^0, v_h)$, $\forall\, v_h\in V_h$. }

\subsection{Stability and convergence}

We first discuss the stability of the ensemble algorithm \eqref{S0}. 
\begin{theorem}\label{th:stability}
Suppose $f_j\in L^2(H^{-1}(D); 0,T)$ 
{   
and conditions (i) and (ii) are satisfied, 
the ensemble scheme \eqref{S0} is stable provided that 
\begin{equation}
 \theta>\theta_+.
 \label{eq:cond1}
\end{equation}
}
Furthermore, the numerical solution to (\ref{S0}) satisfies
 \begin{equation}
 \begin{aligned}
\|u_{j,h}^N\|^2+\theta_-\Delta t\|\nabla
 u_{j,h}^{N}\|^2+(\theta-\theta_+)\Delta t\sum_{n=1}^{N}\|\nabla
u_{j,h}^{n}\|^2 \leq C\Delta
t\sum\limits_{n=1}^{N}\|f_j^{n}\|_{-1}^2\\
+C\Delta t\|\nabla
 u_{j,h}^0\|^2+\|u_{j,h}^0\|^2,
 \end{aligned}
  \label{eq:stability}
 \end{equation}
where $C$  is a generic positive constant independent of
$J$, $h$ and $\Delta t$.
 \end{theorem}

\begin{proof}
Taking $v_h=u_{j,h}^{n+1}$ in (\ref{S0}), we have
\begin{equation*}
\begin{aligned}
\frac{1}{\Delta
t}(u_{j,h}^{n+1}-u_{j,h}^n,u_{j,h}^{n+1})+(\overline{a}^{n+1}\nabla
u_{j,h}^{n+1},\nabla u_{j,h}^{n+1})
&+((a_j^{n+1}-\overline{a}^{n+1})\nabla u_{j,h}^n,\nabla u_{j,h}^{n+1})\\
&=(f_j^{n+1},u_{j,h}^{n+1}).
\end{aligned}
\end{equation*}
Multiplying both sides by $\Delta t$, and using the polarization identity and coercivity of $\overline{a}^{n+1}$, we get
\begin{equation}
\begin{aligned}
\frac{1}{2}\|u_{j,h}^{n+1}\|^2 -\frac{1}{2}\|u_{j,h}^n\|^2 & + \frac{1}{2}\|u_{j,h}^{n+1}-u_{j,h}^n\|^2
+\Delta t\, \theta \|\nabla u_{j,h}^{n+1}\|^2 \\
&\leq - \Delta t((a_j^{n+1}-\overline{a}^{n+1})\nabla u_{j,h}^n,\nabla
u_{j,h}^{n+1})+\Delta t(f_j^{n+1},u_{j,h}^{n+1}). \label{S1}
\end{aligned}
\end{equation}
By the Cauchy-Schwarz and Young's inequalities, we have, for some $\mu$, $\alpha >0$, 
\begin{eqnarray}
\Delta t \left|((a_j^{n+1}-\overline{a}^{n+1})\nabla u_{j,h}^n,\nabla u_{j,h}^{n+1})\right|
&\leq& \Delta t {  \left|a_j^{n+1}-\overline{a}^{n+1}\right|_\infty} \|\nabla u_{j,h}^n\| \|\nabla u_{j,h}^{n+1}\| \label{e1} \nonumber \\
& \leq& \Delta t {  \left|a_j^{n+1}-\overline{a}^{n+1}\right|_\infty} \left(\frac{1}{2\mu}\|\nabla
u_{j,h}^n\|^2+\frac{\mu}{2}\|\nabla u_{j,h}^{n+1}\|^2\right) 
\end{eqnarray}
and 
\begin{equation}
\Delta t|(f_j^{n+1},u_{j,h}^{n+1})|\leq \Delta
t\|f_j^{n+1}\|_{-1}\|\nabla u_{j,h}^{n+1}\|\leq \frac{\Delta
t}{4\alpha}\|f_j^{n+1}\|_{-1}^2+\alpha\Delta t\|\nabla
u_{j,h}^{n+1}\|^2. \label{e2}
\end{equation}
Substituting (\ref{e1}) and (\ref{e2}) into (\ref{S1}) and
dropping  the non-negative term
$\frac{1}{2}\|u_{j,h}^{n+1}-u_{j,h}^n\|^2$, we get
\begin{eqnarray*}
\frac{1}{2}\left(\|u_{j,h}^{n+1}\|^2-\|u_{j,h}^n\|^2\right)
&+&\Delta t\left[ \theta-\alpha- \left(\frac{\mu}{2}+\frac{1}{2\mu} \right)\left|a_j^{n+1}-\overline{a}^{n+1}\right|_\infty\right] \|\nabla u_{j,h}^{n+1}\|^2\nonumber\\
&+&\frac{\Delta t}{2\mu}\left|a_j^{n+1}-\overline{a}^{n+1}\right|_\infty\left(\|\nabla
u_{j,h}^{n+1}\|^2-\|\nabla u_h^n\|^2\right)\leq \frac{\Delta
t}{4\alpha}\|f_j^{n+1}\|_{-1}^{2}.
\end{eqnarray*}
{Multiplying both sides by $2$}, summing over $n$ from 0 to $N-1$ and taking $\mu=1$ yields 
\begin{eqnarray}
\|u_{j,h}^{N}\|^2
&-&\|u_{j,h}^0\|^2 + 2\Delta t\sum\limits_{n=0}^{N-1}\left(\theta-\alpha-\left|a_j^{n+1}-\overline{a}^{n+1}\right|_\infty\right)\|\nabla u_{j,h}^{n+1}\|^2 \nonumber\\
&+&\Delta t \sum\limits_{n=0}^{N-1}
\left|a_j^{n+1}-\overline{a}^{n+1}\right|_\infty \left(\|\nabla
u_{j,h}^{n+1}\|^2-\|\nabla u_{j,h}^n\|^2\right) \leq \frac{\Delta
t}{2\alpha}\sum\limits_{n=0}^{N-1}\|f_j^{n+1}\|_{-1}^2.
\label{eq:stab_1}
\end{eqnarray}
We then choose $\alpha= (\theta-{   \left|a_j^{n+1}-\overline{a}^{n+1}\right| }_\infty)/2$ and use the conditions \eqref{eq:cond_bound} and \eqref{eq:cond1}, and obtain 
 \begin{equation*}
 \begin{aligned}
 \|u_{j,h}^N\|^2+\theta_-\Delta t\|\nabla u_{j,h}^{N}\|^2+(\theta-\theta_+)\Delta t\sum_{n=0}^{N-1}\|\nabla
u_{j,h}^{n+1}\|^2 \leq
 {  \frac{\Delta t}{\theta-\theta_{+}}}\sum\limits_{n=0}^{N-1}\|f_j^{n+1}\|_{-1}^2\\
 +{  \theta_{-}}\Delta t\|\nabla
 u_{j,h}^0\|^2+\|u_{j,h}^0\|^2.
 \end{aligned}
 \end{equation*}
 This completes the proof.
\end{proof}
\begin{remark}
\label{rem:cond1}
The stability condition \eqref{eq:cond1} requires, {for $\{a_j\}_{j=1}^J$}, the deviation of {  $a_j$} from the ensemble average $\overline{a}$ to be less than the coercivity constant $\theta$. If this is not the case, one might divide the ensemble into smaller groups so that the stability condition holds in each of them and the algorithm is stable and applicable.
\end{remark}

Next, we estimate the approximation error of the ensemble algorithm \eqref{S0}. We assume the exact solution of the PDEs is smooth enough, in particular,
\begin{equation*}
u_j\in L^2(H_0^1(D) \cap H^{l+1}(D); 0,T)
\cap H^1(H^{l+1}(D); 0,T) \cap H^2(L^2(D); 0,T).
\end{equation*}
\begin{theorem}
\label{th:error}
 Let $u_j^n$ and $u_{j, h}^n$  be the solutions of equations
\eqref{Eq6} and \eqref{S0} at time $t_n$, respectively.
Assume 
{  $f_j \in
L^2({H}^{-1}(D); 0,T)$ and conditions (i) and (ii) hold.}
{Then} there exists a generic constant $C>0$ independent
of $J$, $h$ and $\Delta t$ such that  
\begin{eqnarray}
\|u_j^N-u_{j,h}^N\|^2+\theta_-\Delta
t\|\nabla(u_j^N-u_{j,h}^{N})\|^2+{   \big(\theta-\theta_+\big)}&\Delta&
t\sum\limits_{n=1}^{N} \|\nabla(u_j^{n}-u_{j,h}^{n})\|^2
\nonumber\\
&&\leq C(\Delta t^2+ h^{2l}){  ,} 
\label{eq:th_err}
\end{eqnarray}
{  
provided that the stability condition \eqref{eq:cond1} holds, that is, $\theta>\theta_+.$
}
\end{theorem}
\begin{proof}
In order to estimate the approximation error of \eqref{S0}, we first find the error equation.
Evaluating the equation (\ref{Eq6}) at $t=t_{n+1}$, tested by $v_h\in V_h$,
yields
\begin{eqnarray}
\frac{1}{\Delta t}(u_j^{n+1}-u_j^{n},v_h)+(a_j^{n+1}\nabla
u_j^{n+1},\nabla v_h)=(f_j^{n+1},v_h)-(r^{n+1}_j,v_h),\label{Eq3}
\end{eqnarray}
where $r^{n+1}_j=u_{j,t}^{n+1}-\frac{u_j^{n+1}-u_j^n}{\Delta t}$.
Denoted by $e_j^n:=u_j^n-u_{j,h}^n$ the approximation error of the $j$-th simulation at time $t_n$.
Subtracting \eqref{S0} from \eqref{Eq3}, we have 
\begin{eqnarray}
\frac{1}{\Delta t}(e_{j}^{n+1}-e_{j}^n,v_h)&+&(\overline{a}^{n+1}\nabla
e_{j}^{n+1},\nabla v_h)+((a_j^{n+1}-\overline{a}^{n+1})\nabla e_{j}^n,\nabla
v_h)\nonumber\\
&+&((a_j^{n+1}-\overline{a}^{n+1})\nabla( u_j^{n+1}-u_j^n),\nabla v_h)
+(r^{n+1}_j,v_h)=0.
\label{eq:error}
\end{eqnarray}
Now we decompose the error as
$$
e_j^n=(u_j^n-P_h(u_j^n))-(u_{j,h}^n-P_h(u_j^n))=\rho_j^n-\phi_{j,h}^n,
$$
where $\rho_j^n= u_j^n-P_h(u_j^n)$ and $\phi_{j,h}^n= u_{j,h}^n-P_h(u_j^n)$ with $P_h(u_j^n)$ the $L^2$ projection of $u_j^n$ {onto} $V_h$, that is, $(\rho_j^n, v_h)=0$ for any $v_h\in V_h$.
We then use the decomposition in \eqref{eq:error} and obtain
\begin{eqnarray}
\frac{1}{\Delta t}(\phi_{j, h}^{n+1}
-\phi_{j, h}^n,v_h)&+&(\overline{a}^{n+1}\nabla
\phi_{j, h}^{n+1},\nabla v_h)+((a_j^{n+1}-\overline{a}^{n+1})\nabla
\phi_{j, h}^n,\nabla v_h) \nonumber\\
&=&(\overline{a}^{n+1}\nabla \rho_{j}^{n+1},\nabla v_h)
+((a_j^{n+1}-\overline{a}^{n+1})\nabla \rho_{j}^n,\nabla v_h) \nonumber\\
&+& ((a_j^{n+1}-\overline{a}^{n+1})\nabla( u_j^{n+1}-u_j^n),\nabla
v_h)+(r_j^{n+1},v_h).
\label{eq:error2}
\end{eqnarray}
Letting $v_h=\phi_{j,h}^{n+1}$ and using the polarization identity {   and coercivity \eqref{eq:cond_coercivity}}, we have
\begin{eqnarray}
\frac{1}{2\Delta t}\big(\|\phi_{j,h}^{n+1}\|^2
&-&\|\phi_{j,h}^n\|^2+\|\phi_{j,h}^{n+1}-\phi_{j,h}^n\|^2\big)+\theta\|\nabla\phi_{j,h}^{n+1}\|^2
\nonumber\\\
&\leq&
|((a_j^{n+1}-\overline{a}^{n+1})\nabla \phi_{j,h}^n,\nabla \phi_{j,h}^{n+1})| \nonumber\\
&+& |(\overline{a}^{n+1}\nabla \rho_{j}^{n+1},\nabla \phi_{j,h}^{n+1})|+|((a_j^{n+1}-\overline{a}^{n+1})\nabla \rho_{j}^n, \nabla\phi_{j,h}^{n+1})| \nonumber \\
&+& |((a_j^{n+1}-\overline{a}^{n+1})\nabla( u_j^{n+1}-u_j^n),\nabla
\phi_{j,h}^{n+1})|+|(r_j^{n+1}, \phi_{j,h}^{n+1})|. \label{error3}
\end{eqnarray}
The Cauchy-Schwarz and Young's inequalities applied to the RHS terms of (\ref{error3}) give, for any positive constants $\beta_i$ ($i= 0, \ldots, 3$), that
\begin{equation*}
|((a_j^{n+1}-\overline{a}^{n+1})\nabla \phi_{j,h}^n,\nabla
\phi_{j,h}^{n+1})|
\leq
 {   \left|a_j^{n+1}-\overline{a}^{n+1}\right| }_{\infty} \left(\frac{\|\nabla
\phi_{j,h}^n\|^2}{2}+\frac{\|\nabla\phi_{j,h}^{n+1}\|^2}{2}\right),
\end{equation*}
\begin{equation*}
|(\overline{a}^{n+1}\nabla \rho_{j}^{n+1},\nabla \phi_{j,h}^{n+1})|
\leq
{  |\overline{a}^{n+1}|_{\infty}} \left(\frac{\|\nabla\rho_j^{n+1}\|^2}{2\beta_0}+\frac{\beta_0\|\nabla\phi_{j,h}^{n+1}\|^2}{2}\right),
\end{equation*}
\begin{equation*}
|((a_j^{n+1}-\overline{a}^{n+1})\nabla \rho_{j}^n,
\nabla\phi_{j,h}^{n+1})|
\leq
{   \left|a_j^{n+1}-\overline{a}^{n+1}\right| }_{\infty} \left(\frac{\|\nabla\rho_j^{n}\|^2}{2\beta_1}+\frac{\beta_1\|\nabla\phi_{j,h}^{n+1}\|^2}{2}\right),
\end{equation*}
\begin{equation*}
\begin{aligned}
&|((a_j^{n+1}-\overline{a}^{n+1})\nabla( u_j^{n+1}-u_j^n),\nabla \phi_{j,h}^{n+1})|
\\
&\hspace{2cm}\leq
{   \left|a_j^{n+1}-\overline{a}^{n+1}\right| }_{\infty}\left(\frac{\|\nabla u_j^{n+1}-\nabla
u^n_j\|^2}{2\beta_2}+\frac{\beta_2\|\nabla \phi_{j,h}^{n+1}\|^2}{2}\right),
\end{aligned}
\end{equation*}
\begin{equation*}
|(r_j^{n+1}, \phi_{j,h}^{n+1})|
\leq
\left(\frac{\|r_j^{n+1}\|_{-1}^2}{2\beta_3}+\frac{\beta_3\|\nabla\phi_{j,h}^{n+1}\|^2}{2}\right).
\end{equation*}
Using {the} above inequalities in (\ref{error3}) and dropping the non-negative
term, $\frac{1}{2\Delta t}\|\phi_{j,h}^{n+1}-\phi_{j,h}^{n}\|^2$ on the left hand side (LHS), we get
\begin{eqnarray*}
&&\frac{1}{2\Delta t}(\|\phi_{j,h}^{n+1}\|^2-\|\phi_{j,h}^n\|^2) +\frac{{   \left|a_j^{n+1}-\overline{a}^{n+1}\right| }_\infty}{2}(\|\nabla\phi_{j,h}^{n+1}\|^2-\|\nabla\phi_{j,h}^n\|^2)
\nonumber\\
&&\qquad+\left(\theta-\frac{\beta_0}{2}{  |\overline{a}^{n+1}|_{\infty}}-\frac{\beta_1+\beta_2+2}{2}{   \left|a_j^{n+1}-\overline{a}^{n+1}\right| }_\infty
-\frac{\beta_3}{2}\right)\|\nabla\phi_{j,h}^{n+1}\|^2
\nonumber\\
&&\quad\leq
\left(\frac{{  |\overline{a}^{n+1}|_{\infty}}}{2\beta_0}\|\nabla\rho_j^{n+1}\|^2
+ \frac{{   \left|a_j^{n+1}-\overline{a}^{n+1}\right| }_\infty}{2\beta_1}\|\nabla\rho_j^n\|^2
+ \frac{{   \left|a_j^{n+1}-\overline{a}^{n+1}\right| }_\infty}{2\beta_2}\|\nabla u_j^{n+1}-\nabla u_j^n\|^2
\right.
\nonumber \\
&&\qquad \left.
+\frac{\|r_j^{n+1}\|_{-1}^2}{2\beta_3}\right).
\label{error4}
\end{eqnarray*}
Selecting {$\beta_0=\frac{\delta|a_j^{n+1}-\overline{a}^{n+1}|_\infty}{2|\overline{a}^{n+1}|_{\infty}}$,
$\beta_1=\beta_2= \frac{\delta}{2}$, and 
$\beta_3=\frac{\delta{   \left|a_j^{n+1}-\overline{a}^{n+1}\right| }_\infty}{2}$ for some positive $\delta$}, yields
\begin{eqnarray}
&\frac{1}{2\Delta t}(\|\phi_{j,h}^{n+1}\|^2-\|\phi_{j,h}^n\|^2)
+\frac{{   \left|a_j^{n+1}-\overline{a}^{n+1}\right| }_\infty}{2}(\|\nabla\phi_{j,h}^{n+1}\|^2 -\|\nabla\phi_{j,h}^n\|^2)
\nonumber\\
&+{  \left[\theta-(1+\delta){   \left|a_j^{n+1}-\overline{a}^{n+1}\right| }_\infty\right]}\|\nabla\phi_{j,h}^{n+1}\|^2
\nonumber\\
&\leq {  \frac{|\overline{a}^{n+1}|_{\infty}^2}{\delta{   \left|a_j^{n+1}-\overline{a}^{n+1}\right| }_\infty}} 
\|\nabla\rho_j^{n+1}\|^2
+{  \frac{{   \left|a_j^{n+1}-\overline{a}^{n+1}\right| }_\infty }{\delta}}\|\nabla\rho_j^n\|^2\nonumber\\
&+{  \frac{{   \left|a_j^{n+1}-\overline{a}^{n+1}\right| }_\infty}{\delta}}\|\nabla u_j^{n+1}-\nabla u_j^n\|^2+\frac{\|r_j^{n+1}\|_{-1}^2}{{  \delta{   \left|a_j^{n+1}-\overline{a}^{n+1}\right| }_\infty}}.
\label{error5}
\end{eqnarray}
{Taking $\delta= \frac{\theta-\theta_{+}}{2\theta_{+}}$, we have $\theta-(1+\delta) \left|a_j^{n+1}-\overline{a}^{n+1}\right|_\infty> \frac{\theta-\theta_{+}}{2}>0$ based on the stability condition \eqref{eq:cond1} and the upper bound in condition \eqref{eq:cond_bound}.
}
In the last two terms on the RHS of (\ref{error5}), note that
\begin{eqnarray*}
\|\nabla u_j^{n+1}-\nabla u_j^n\|^2=\int_D|\nabla u_j^{n+1}-\nabla
u_j^n|^2dx=\int_D\left|\int_{t_n}^{t_{n+1}}(\nabla u_j)_t \,dt\right|^2\,d\bx\\
\leq \Delta t \int_{t_n}^{t_{n+1}} \int_D|(\nabla u_j)_t|^2\,d\bx\,dt=\Delta
t\|\nabla u_{j,t}\|^2_{L^2(L^2(D); t_{n}, t_{n+1})}{  .}
\end{eqnarray*}
{By the integral form of Taylor's theorem}
\begin{equation*}
	u_j^n=u_j^{n+1}-\Delta t\, u_{j, t}^{n+1}{  -}\int_{t_n}^{t_{n+1}}u_{j, tt}(\cdot, s)(t_n-s)\,ds,
\end{equation*}
we have
\begin{eqnarray*}
\|r_j^{n+1}\|&=&\frac{1}{\Delta t}\left\|\int_{t_n}^{t_{n+1}}u_{j,tt}(\cdot, s)(s-t_n)\,ds\right\|
\leq \int_{t_n}^{t_{n+1}}\|u_{j,tt}(\cdot, s)\|\cdot 1\,ds \\
&\leq&{  \left[ \int_{t_n}^{t_{n+1}}\|u_{j,tt}(\cdot, s)\|^2\,ds \right]^{1/2}\left(\int_{t_n}^{t_{n+1}}1^2\,ds\right)^{1/2}}\\
&\leq& \sqrt{\Delta
t}\|u_{j,tt}\|_{L^2(L^2(D); t_n, t_{n+1})}
\end{eqnarray*}
and
\begin{equation*}
\|r_j^{n+1}\|_{-1}^2\leq C\|r_j^{n+1}\|^2\leq C\Delta t\|u_{j,tt}\|_{L^2(L^2(D); t_n, t_{n+1})}^2.
\end{equation*}
Substituting these inequalities in (\ref{error5}), considering the uniform bounds of $\left|a_j-\overline{a}\right|_{\infty}$ given in \eqref{eq:cond_bound}, multiplying both sides of (\ref{error5}) by $2\Delta t$, and summing over $n$, we get
{  
\begin{eqnarray*}
&&\|\phi_{j,h}^N\|^2
+\left(\theta-\theta_{+}\right)\Delta t\sum\limits_{n=1}^N\|\nabla\phi_{j,h}^{n}\|^2
+\theta_{-} \Delta t \|\nabla\phi_{j,h}^N\|^2  \\ \nonumber
&&\quad\leq
\frac{4\Delta t \theta_{+}}{\theta-\theta_{+}} \sum_{n=0}^{N-1}\left(
\frac{1}{\theta_{-}} |\overline{a}^{n+1}|_\infty^2 \|\nabla \rho_j^{n+1}\|^2
+ \theta_{+} \|\nabla \rho_j^n\|^2 
+ \theta_{+} \Delta t\|\nabla u_{j,t}\|_{L^2(L^2(D); t_n, t_{n+1})}^2 
\right.    \\ \nonumber
&&\quad+\left. \frac{C}{\theta_{-}} \Delta t \|u_{j,tt}\|_{L^2(L^2(D); t_n, t_{n+1})}^2  \right),
\end{eqnarray*}
}
where we used the assumption that {  $u_{j,h}^0=P_h(u_j^0)$}, thus $\|\phi_{j,h}^0\|=\|\nabla\phi_{j,h}^0\|=0$.
By the regularity assumption and  standard finite element estimates of the $L^2$ projection error in the
$H^1$ norm {(see, e.g., Section 4.4 in \cite{brennermathematical}), 
namely, for any $u_j^n\in H^{l+1}(D)\cap H_0^1(D),$
\begin{equation}
\|\nabla \rho_j^n\| = \|\nabla\big(P_h(u_j^n)-u_j^n\big)\|^2\leq Ch^{2l}\|u_j^n\|_{l+1}^2,
\end{equation}
} we have
\begin{equation}
\|\phi_{j,h}^N\|^2+{  \left( \theta-\theta_{+}\right)}\Delta
t\sum\limits_{n=0}^{N-1}\|\nabla\phi_{j,h}^{n+1}\|^2
+\theta_{-} \Delta t \|\nabla\phi_{j,h}^N\|^2 \leq
C(\Delta t^2+h^{2l}),
\end{equation}
where $C$ is a generic constant independent of the time step $\Delta t$ and mesh size $h$.
By the triangle inequality, we have the error estimate \eqref{eq:th_err}.
\end{proof}

The proposed ensemble scheme can be easily extended to more general parabolic equations such as those with random coefficients.
Next we discuss the ensemble solution to unsteady random diffusion equations.

\section{The ensemble scheme of parabolic equations with random coefficients}
\label{sec:random}
We consider numerical simulations of an unsteady heat equation for a spatially {   and temporally} varying medium, in
the absence of convection. That is, to find a random function,
$u: \Omega\times \overline{D}\times [0, T] \rightarrow \mathbb{R}$
satisfying a.s.,
\begin{equation}
\left\{
\begin{aligned}
&u_t-\nabla\cdot[{  a(\omega, \bx,t)}\nabla u]=f(\omega, \bx,  t), &\text{ in } \Omega\times D\times[0, T]{  ,} \\
 &u(\omega, \bx, t) = {  g(\omega, \bx, t)}, &\text{ on } \Omega\times \partial D\times [0, T]{  ,} \\
 &u(\omega, \bx, 0) = {  u^0(\omega,\bx)}, &\text{ in }\Omega \times D{  ,}
 \end{aligned}
 \right.
 \label{Eq11}
 \end{equation}
where $D$ is a bounded Lipschitz domain in $\mathbb{R}^d$, 
$\Omega$ is the set of outcomes in the complete probability space, 
diffusion coefficient $a$, 
source term $f$ and 
{  
boundary condition $g$: $\Omega\times D\times[0,T] \rightarrow \mathbb{R}$, 
and initial condition $u^0$: $\Omega\times D \rightarrow \mathbb{R}$ are random fields 
}
with continuous and
bounded covariance functions.

As a first step of the investigations on the ensemble method to random PDEs, we choose the Monte Carlo method for random sampling because it is nonintrusive, easy to implement, and its convergence is independent of the dimension of the uncertain model parameters.
However, other methods like Quasi Monte Carlo, Latin Hypercube Sampling,  Stochastic Collocation (see, e.g.,  \cite{niederreiter1992random,helton2003latin,babuvska2005solving,xiu2005high,mathelin2005stochastic,ganapathysubramanian2007sparse,gunzburger2014stochastic,zhu2017multi} and references therein) are applicable, too.
%
When the Monte Carlo method is applied, a large number of samples are randomly selected first, then a group of independent simulations needs to be implemented in order to quantify the underlying stochastic information of the problem.
To improve its computational efficiency, we propose an {\em ensemble-based Monte-Carlo (EMC) method} for the { purpose of} uncertainty quantification.
The method consists of the following steps:\par
\begin{enumerate}
\item Choose a set of random samples for the random medium coefficient, source term, boundary and initial conditions:
{  
$a_j\equiv a(\omega_j,\cdot,\cdot)$, $f_j \equiv f(\omega_j,\cdot, \cdot)$,  $g_j \equiv g(\omega_j, \cdot, \cdot)$, and $u_j^0\equiv u^0(\omega_j,\cdot)$
} 
for $j= 1, \ldots, J$.
Note that the corresponding solutions $u(\omega_j, \cdot, \cdot)$ are independent, identically
distributed (i.i.d.).
\item Use the uniform time partition on $[0, T]$ with the step size $\Delta t = T/N$. Define $u_{j}^n = u(\omega_j, \bx, t_n)$ and $\overline{a}^n = \frac{1}{J}\sum_{j=1}^J a(\omega_j, \bx, t_n)$. 
{  
For $j=1, \ldots, J$ and $n= 0, \ldots, N-1$, one finds $u_j^{n+1}$ satisfying the ensemble scheme \eqref{e0}. In practice, an appropriate finite element space on the mesh $\mathcal{T}_h$ could be chosen, on which one finds the finite dimensional approximation $u_h(\omega_j, \cdot, \cdot)$.
}
 \item {  Approximate $E[u]$ by the EMC sample average $\frac{1}{J}\sum_{j=1}^J u_h(\omega_j, \cdot, \cdot)$.}
If a quantity of interest $Q(u)$ is given, one analyzes the outputs from the ensemble simulations, $Q\left(u_h(\omega_1, \cdot, \cdot)\right), \ldots, Q\left(u_h(\omega_J, \cdot, \cdot)\right)$, to extract the stochastic information. 
\end{enumerate}

{  It is seen that the EMC method naturally synthesizes the ensemble-based time-stepping algorithm \eqref{e0} with the Monte-Carlo random sampling approach. It keeps the same advantage of the ensemble algorithm when applying to the deterministic PDEs: all the simulations on the selected samples would share a single coefficient matrix at each time step, thus one only needs to solve a linear system with multiple RHS vectors, which leads to the reduction of computational cost.}
Next, we derive some numerical analysis for the proposed method.

\subsection{Stability and convergence}

{  Similar to the deterministic case, we consider problems with homogeneous boundary conditions in the following analysis, which could be extended to the inhomogeneous cases by means of the method of shifting.}
 Choose the same finite element space $V_h$ as defined in Section \ref{sec:deterministic}.
Denote $u_{j, h}^n = u_h(\omega_j, \bx, t_n)$.
For the $j$-th ensemble member and for {$n=0,\cdots, N-1$},
find an approximation solution $u_{j,h}^{n+1}\in V_h$ such that
\begin{eqnarray}
\left(\frac{u_{j,h}^{n+1}-u_{j,h}^n}{\Delta t},v_h\right)+({  \overline{a}^{n+1}}\nabla
u_{j,h}^{n+1},\nabla v_h)&+&({  (a_j^{n+1}-\overline{a}^{n+1})}\nabla u_{j,h}^n,\nabla
v_h)\nonumber \\
&=&(f_{j}^{n+1},v_h), \quad \forall v_h\in V_h
\label{S}
\end{eqnarray}
{  
with the initial condition $u_{j,h}^0\in V_h$ satisfying $(u_{j, h}^0, v_h) = (u_j^0, v_h)$, $\forall\, v_h\in V_h$. 
Although \eqref{S} has the same form as \eqref{S0}, we still present it here because $u_{j, h}^n$ in \eqref{S} changes from a real-valued function to a random variable. 
}

{   
Suppose the following two conditions are valid: 
\begin{itemize}
\item[(iii)] There exists a positive constant $\theta$ such that, for any $t\in [0,T]$, 
\begin{equation}
P\{\omega\in\Omega;\min\limits_{x\in
\overline{D}}{  a(\omega, \bx,t)}>\theta\}=1.
 \label{eq:cond_coercivity_rand}
\end{equation}
\item[(iv)]There exist positive constants $\theta_{-}$ and $\theta_{+}$ such that, for any $t\in [0, T]$, 
\begin{equation}
P\{\omega_j \in \Omega; \theta_{-}\leq |a(\omega_j, \bx, t)-\overline{a}|_\infty\leq\theta_{+}\}=1.
 \label{eq:cond_bound_rand}
\end{equation}
\end{itemize}
Here, condition (iii) guarantees the uniform coercivity a.s.; condition (iv) gives the uniform bounds of the distance from coefficient $a(\omega_j, \bx, t)$ to the ensemble average $\overline{a}=\frac{1}{J}\sum_{j=1}^J a(\omega_j, \bx, t)$ a.s. 
}

Theorem \ref{th:stability} together with the property of expectation lead to the
following stability analysis for the finite element solution $u_{j,h}^n$:
\begin{theorem} \label{th:stability_R}
Suppose $f_j\in {\widetilde{L}}^2({H}^{-1}(D); 0,T)$ and 
conditions (iii) and (iv) are satisfied, the finite element solution $u_{j,h}^n$ to \eqref{S} is stable provided 
 \begin{equation}
\theta>\theta_+. 
\label{sta:con_rand1}
\end{equation}
Especially, for any $\Delta t>0$, the solution satisfies
\begin{equation}
\begin{aligned}
E\left[\|u_{j,h}^N\|^2\right]&+\theta_-\Delta tE\left[\|\nabla
u_{j,h}^{N}\|^2\right]+(\theta-\theta_+)\Delta
t\sum\limits_{n=1}^{N}E\left[\|\nabla
 u_{j,h}^{n}\|^2\right]\\
&\leq
C\Delta t\sum\limits_{n=1}^{N}E\left[\|f_j^{n}\|_{-1}^2\right]
 +C\Delta t E\left[\|\nabla u_{j,h}^0\|^2\right] 
 +E\left[\|u_{j,h}^0\|^2\right], 
 \end{aligned}
 \end{equation}
 where $C$ is a generic constant independent of $J$, $h$ and $\Delta t$.
 \end{theorem}

The stability condition \eqref{sta:con_rand1} restricts the deviation of random diffusion coefficients from the  ensemble average. Similar to the deterministic case (see Remark \ref{rem:cond1}), if it does not hold, one might separate the entire ensemble into smaller groups to ensure that  \eqref{sta:con_rand1} is true for each of the small groups, then the EMC method will be applicable to all the groups.

{The full-discrete EMC approximation is defined to be $\Psi_h^n \equiv \frac{1}{J}\sum_{j=1}^J u_{j,h}^n$. Next, we will derive an estimate for $E[u^n]-\Psi_h^n$ in certain averaged norms. Note that $E[u^n]-\Psi_h^n$}
can be naturally split into two parts:
\begin{eqnarray*}
E[u^n]-\Psi_h^n
&=& \left( E[u_j^n]-E[u_{j, h}^n] \right)+\left( E[u_{j, h}^n]-\Psi_h^n \right)\\
&=&\mathcal{E}_h^n+\mathcal{E}_S^n,
\end{eqnarray*}
where we use $E[u^n] = E[u_j^n]$ in the first equality. The first part, $\mathcal{E}_h^n=E[u_j^n-u_{j, h}^n]$, is related to the finite element discretization error
controlled by the size of the spatial triangulation and time step;
while the second part, $\mathcal{E}_S^n= E[u_{j, h}^n]-\Psi_h^n$, is the statistical error controlled by the number of realizations. 
{In the following, we will analyze $\mathcal{E}_h^n$ in Theorem \ref{th:E_h}, bound $\mathcal{E}_S^n$ in Theorem \ref{th:E_S}, and obtain an error estimate of the EMC approximation in Theorem \ref{th:rand_error}.}

For $\mathcal{E}_h^n$, we have the following estimate:
\begin{theorem}\label{th:E_h}
 Let $u_j^n$ be the solution to equation (\ref{Eq11}) when $\omega= \omega_j$ and $t=t_n$, and $u_{j, h}^n$ be the solution to (\ref{S}). Suppose $u^0_j \in\widetilde{L}^2({H}_0^1(D)\cap {H}^{l+1}(D))$,
 $f_j \in \widetilde{L}^2({H}^{-1}(D); 0,T)$. 
Under conditions (iii) and (iv), there exists a generic constant $C>0$ independent of
$J$, $h$ and $\Delta t$ such that
\begin{eqnarray}
E\left[\|u_j^N-u_{j,h}^N\|^2\right]+{   \big(\theta-\theta_+\big)}\Delta
t\sum\limits_{n=1}^N E\left[\|\nabla(u_j^n-u_{j,h}^{n})\|^2\right]\nonumber\\
+\theta_{-} \Delta tE\left[\|\nabla(u_j^N-u_{j,h}^{N})\|^2\right] \leq C(\Delta
t^2+h^{2l}),
\end{eqnarray}
provided that {   the stability condition \eqref{sta:con_rand1} holds. }
\end{theorem}
\begin{proof}
{  
The conclusion follows Theorem \ref{th:error} after applying the expectation on \eqref{eq:th_err}.}
\end{proof}

With the standard error estimate of the Monte Carlo
method (e.g., see \cite{liu2013discontinuous}), the statistical error $\mathcal{E}_S^n$
can be bounded as follows:

\begin{theorem}\label{th:E_S}
Suppose conditions (iii) and (iv), and the stability condition \eqref{sta:con_rand1} hold, $f_j\in {\widetilde{L}}^2({H}^{-1}(D); 0,T)$ and $u^0_j \in\widetilde{L}^2({H}_0^1(D)\cap {H}^{l+1}(D))$, 
then there is a generic constant $C>0$ independent of $J$, $h$ and $\Delta t$ such
that
\begin{equation}
\begin{aligned}
E\left[\|\mathcal{E}_S^N\|^2\right]+\theta_{-}\Delta tE\left[\|\nabla
\mathcal{E}_S^N\|^2\right]+(\theta-\theta_+){  \Delta t\sum\limits_{n=1}^{N}E\left[\|\nabla\mathcal{E}_S^n\|^2\right]}\\
\leq
\frac{C}{J}
 \Big(\Delta t\sum\limits_{n=1}^{N}E\left[\|f_j^{n}\|_{-1}^2\right]
 +\Delta t E\left[\|\nabla
 u_{j,h}^0\|^2\right]+E\left[\|u_{j,h}^0\|^2\right]\Big).
 \end{aligned}
 \label{eq:rand_serr1}
\end{equation}
\end{theorem}

\begin{proof}
{We first estimate $E[\|\nabla\mathcal{E}_S^n\|]$, define $\langle u_h^n, u_h^n\rangle:=(\nabla u_h^n,\nabla u_h^n)$, then we have}
\begin{eqnarray*}
E\left[{  \|\nabla{\mathcal{E}_S^n}\|^2}\right]&=&E\Big[\Big\langle\frac{1}{J}\sum\limits_{i=1}^J(E[u_h^n]-u_{i, h}^n),\frac{1}{J}\sum\limits_{j=1}^J(E[u_h^n]-u_{j, h}^n) \Big\rangle\Big]\\
&=&\frac{1}{J^2}\sum\limits_{i=1}^J\sum\limits_{j=1}^J E \left[ \langle E[u_h^n]-u_{i, h}^n,E[u_h^n]-u_{j, h}^n\rangle \right]\\
&=&\frac{1}{J^2}\sum\limits_{j=1}^JE \left[\langle
E[u_h^n]-u_{j, h}^n,E[u_h^n]-u_{j, h}^n \rangle \right].
\end{eqnarray*}
The last equality is due to the fact that $u_h^n(\omega_1,{  \cdot}), \ldots, u_h^n(\omega_J,{  \cdot})$
are i.i.d., and thus the expected value of $\langle
E[u_h^n]-u_{i, h}^n,E[u_h^n]-u_{j, h}^n \rangle$
is a zero for $i\neq j$. We now expand the quantity $\langle
E[u_h^n]-u_{j, h}^n, E[u_h^n]-u_{j, h}^n\rangle$
and use the fact that $E[u_h^n]=E[u_{j, h}^n]$
and $E[(u_h^n)^2]=E[(u_{j, h}^n)^2]$ to obtain
$$
E\left[{  \|\nabla\mathcal{E}_S^n\|^2}\right]=-\frac{1}{J}{  \|\nabla E[u_{j,h}^n]\|^2}+\frac{1}{J}E\left[{  \|\nabla u_{j,h}^n\|^2}\right].
$$
Therefore, we have
$$
E\left[{  \|\nabla\mathcal{E}_S^n\|^2}\right]\leq \frac{1}{J}E\left[{  \|\nabla u_{j,h}^n\|^2}\right].
$$
By Theorem \ref{th:stability_R}, we get
\begin{eqnarray*}
{   (\theta-\theta_+) \Delta t\sum\limits_{n=1}^{N}E \left[\|\nabla u_{j,h}^n\|^2 \right]}
&\leq&  C\Delta t\sum\limits_{n=1}^{N}E\left[\|f_j^{n}\|_{-1}^2\right]
\\ \nonumber
&+&C\Delta t E\left[\|\nabla
 u_{j,h}^0\|^2\right]+E\left[\|u_{j,h}^0\|^2\right].
\end{eqnarray*}
The other terms on the LHS of \eqref{eq:rand_serr1} involving $E[\|\mathcal{E}_S^N\|^2]$ and $E[\|\nabla \mathcal{E}_S^N\|^2]$ can be treated in the same manner. This completes the proof.
\end{proof}

The combination of error contributions from the finite element approximation and Monte Carlo sampling yields a bound for {the EMC approximation error in the following sense}:
\begin{theorem}
\label{th:rand_error} 
 For the given source function  $f_{  j}\in
\widetilde{L}^2({H}^{-1}(D); 0,T)$ and
$u^0_{  j}\in\widetilde{L}^2({H}_0^1(D)\cap {H}^{l+1}(D))$. 
Under conditions (iii) and (iv), and suppose the stability condition \eqref{sta:con_rand1} is satisfied, that is, $\theta>\theta_+$, then there holds
\begin{eqnarray}
E\left[\|E[u^N]-\Psi_h^N\|^2\right]
&+&\theta_- \Delta t E\left[\|\nabla(E[u^N]-\Psi_h^N)\|^2\right]\nonumber\\
&+&{  \big(\theta-\theta_+\big)\Delta t\sum\limits_{n=1}^{N}E\left[\|\nabla(E[u^n]-\Psi_h^n)\|^2\right]}\nonumber\\
&\leq& \frac{1}{J}\left(C\Delta t\sum\limits_{n=1}^NE\left[\|f_j^n\|_{-1}^2\right]+C\Delta
t E\left[\|\nabla u_{j,h}^0\|^2\right]+E\left[\|u_{j,h}^0\|^2\right]\right)\nonumber\\
&&+C(\Delta
t^2+h^{2l}),
\label{eq:rand1}
\end{eqnarray}
 where $C>0$ is a constant independent of $J$, $h$ and $\Delta t$.
\end{theorem}

\begin{proof}
Consider the first term on the LHS of \eqref{eq:rand1}. By the triangle and Young's inequality, we {have}
\begin{equation*}
E\left[\|E[u^N]-\Psi_h^N\|^2\right] \leq 2\left( E\left[ \|E[u^N]-E[u_h^N]\|^2 \right] + E\left[ \|E[u_h^N]- \Psi_h^N\|^2 \right] \right).
\end{equation*}
Applying Jensen's inequality to the first term on the RHS of the above inequality, we have
\begin{equation*}
E\left[ \|E[u^N]-E[u_h^N]\|^2\right] \leq E\left[ E[\|u^N- u_h^N\|^2] \right]= E\left[\|u^N- u_h^N\|^2\right].
\end{equation*}
Then the conclusion follows from Theorems \ref{th:E_h}-\ref{th:E_S}.
The other terms on the LHS of \eqref{eq:rand1} can be estimated in a similar manner.
\end{proof}
\section{Numerical experiments}
\label{sec:num}
We present two numerical tests on the ensemble schemes for second-order parabolic PDEs in this section: the first problem is deterministic heat transfer with an {\em a prior} known exact solution, which aims to illustrate {   Theorem \ref{th:error}}; the second problem is random heat transfer without a known exact solution, which {   is used to illustrate Theorem \ref{th:rand_error} and} shows the effectiveness of the ensemble method by comparing the results with those of independent, individual simulations.
\subsection{Deterministic heat transfer}
We first test the numerical performance of the ensemble algorithm on the deterministic second-order parabolic equation \eqref{Eq6}.
A group of simulations is considered, which contains $J=3$ members.
The diffusion coefficient and the exact solution of $j$-th simulation are selected as follows.
\begin{equation*}
\begin{aligned}
a_j(\bx, t)&= 1+(1+\epsilon_j)\sin(t)\sin(xy), \\
u_j(\bx, t)&= (1+\epsilon_j)[\sin(2\pi x)\sin(2\pi y)+\sin(4\pi
t)],
\end{aligned}
\end{equation*}
where $\epsilon_j$ is a perturbation randomly selected from $[0, 1]$, $t\in [0, 1]$ and $(x, y)\in [0, 1]^2$.
The initial condition, Dirichlet boundary condition and source term are chosen to match the exact solution.

The group is simulated by using the ensemble scheme \eqref{S0}.
In the test, the ensemble contains three members with $\epsilon_1 = 0.6207$, $\epsilon_2= 0.1841$, and $\epsilon_3 = 0.2691$.
In order to check the convergence order in time, we use quadratic finite elements, a uniform time partition, and uniformly refine the mesh size $h$ and time step size $\Delta t$ from the initial mesh size $\sqrt{2}/4$ and initial time step size $1/10$.
{   Let the maximum numerical approximation errors in $L^2$ norm be  
$$\mathcal{E}_{L^2}^{j} = \max_{n\in \{1, \ldots, N\}} \|u_j^n - u_{j, h}^n\|$$
and errors in the time average $H^1$ semi-norm be  
$$\mathcal{E}_{H^1}^{j} = \sqrt{ \Delta t\sum_{n=1}^N \|\nabla u_j^n - \nabla u_{j, h}^n\|^2} $$
for $j= 1, 2, 3$, respectively. 
The approximation errors of the ensemble method (denoted by $\mathcal{E}_{L^2}^{E, j}$ and $\mathcal{E}_{H^1}^{E, j}$) are listed in Table \ref{tab:t2u}.}
It is seen that the rate of convergence is nearly linear, which matches our theoretical analysis in  {   Theorem \ref{th:error}}.

\begin{table}[htp]
\centering
{\footnotesize
\caption{Numerical errors of the ensemble simulations. }
\label{tab:t2u}%
\begin{tabular}{|c|c|c|c|c|c|c|}
\hline
$\sqrt{2}/h$& $\mathcal{E}_{L^2}^{E, 1} $ & rate & $\mathcal{E}_{L^2}^{E, 2}$ & rate & $\mathcal{E}_{L^2}^{E, 3}$ & rate \\
\hline
4 & 2.2271$\times 10^{-1}$ &   -  & 2.2168$\times 10^{-1}$ &   - & 2.2177$\times 10^{-1}$ &   - \\
8 & 1.1477$\times 10^{-1}$ &   0.96  & 1.1623$\times 10^{-1}$ &   0.93 & 1.1594$\times 10^{-1}$ &   0.94 \\
16 & 5.9080$\times 10^{-2}$ &   0.96  & 5.9921$\times 10^{-2}$ &   0.96 & 5.9756$\times 10^{-2}$ &   0.96 \\
32 & 3.0007$\times 10^{-2}$ &   0.98  & 3.0445$\times 10^{-2}$ &   0.98 & 3.0359$\times 10^{-2}$ &   0.98\\
\hline
$\sqrt{2}/h$& $\mathcal{E}_{H^1}^{E, 1} $  & rate & $\mathcal{E}_{H^1}^{E, 2} $  & rate & $\mathcal{E}_{H^1}^{E, 3} $  & rate \\
\hline
4 & 1.3678$\times 10^{0}$ &   -  & 1.0922$\times 10^{0}$ &   - & 1.1437$\times 10^{0}$ &   - \\
8 & 4.7311$\times 10^{-1}$ &   1.53  & 4.2423$\times 10^{-1}$ &   1.36 & 4.3280$\times 10^{-1}$ &   1.40 \\
16 & 1.9969$\times 10^{-1}$ &   1.24  & 1.9560$\times 10^{-1}$ &   1.12 & 1.9618$\times 10^{-1}$ &   1.14 \\
32 & 9.5767$\times 10^{-2}$ &   1.06  & 9.6972$\times 10^{-2}$ &   1.01 & 9.6692$\times 10^{-2}$ &   1.02 \\
\hline
\end{tabular}
}
\end{table}

To compare with the individual simulations, we list in Table \ref{tab:t2u2} the numerical errors of independent simulations (denoted by $\mathcal{E}_{L^2}^{I, j}$ and $\mathcal{E}_{H^1}^{I, j}$) in the same computational setting. 
It is observed that the ensemble simulation results in Table \ref{tab:t2u} are close to those obtained from individual simulations in Table \ref{tab:t2u2}. Indeed, the errors are at the same order of magnitude and the convergence rates are almost same.

\begin{table}[htp]
\centering
{\footnotesize
\caption{Numerical errors of the independent simulations. }
\label{tab:t2u2}%
\begin{tabular}{|c|c|c|c|c|c|c|}
\hline
$\sqrt{2}/h$& $\mathcal{E}_{L^2}^{I, 1} $ & rate & $\mathcal{E}_{L^2}^{I, 2}$ & rate & $\mathcal{E}_{L^2}^{I, 3}$ & rate \\
\hline
4 & 2.2206$\times 10^{-1}$ &   -  & 2.2215$\times 10^{-1}$ &   - & 2.2200$\times 10^{-1}$ &   - \\
8 & 1.1469$\times 10^{-1}$ &   0.95  & 1.1629$\times 10^{-1}$ &   0.93 & 1.1597$\times 10^{-1}$ &   0.94 \\
16 & 5.9072$\times 10^{-2}$ &   0.96  & 5.9928$\times 10^{-2}$ &   0.96 & 5.9759$\times 10^{-2}$ &   0.96 \\
32 & 3.0007$\times 10^{-2}$ &   0.98  & 3.0446$\times 10^{-2}$ &   0.98 & 3.0359$\times 10^{-2}$ &   0.98 \\
\hline
$\sqrt{2}/h$& $\mathcal{E}_{H^1}^{I, 1} $  & rate & $\mathcal{E}_{H^1}^{I, 2} $  & rate & $\mathcal{E}_{H^1}^{I, 3} $  & rate \\
\hline
4 & 1.3641$\times 10^{0}$ &   -  & 1.0955$\times 10^{0}$ &   - & 1.1453$\times 10^{0}$ &   - \\
8 & 4.7186$\times 10^{-1}$ &   1.53  & 4.2529$\times 10^{-1}$ &   1.37 & 4.3331$\times 10^{-1}$ &   1.40 \\
16 & 1.9933$\times 10^{-1}$ &   1.24  & 1.9588$\times 10^{-1}$ &   1.12 & 1.9632$\times 10^{-1}$ &   1.14 \\
32 & 9.5677$\times 10^{-2}$ &   1.06  & 9.7041$\times 10^{-2}$ &   1.01 & 9.6726$\times 10^{-2}$ &   1.02 \\
\hline
\end{tabular}
}
\end{table}

\subsection{Random heat transfer}
Next we consider the second-order parabolic equation with a random diffusion
coefficient \eqref{Eq11} on the unit square domain. The test problem is
associated with the zero forcing term $f$, zero initial conditions,
and homogeneous Dirichlet boundary conditions on the top, bottom
and right edges of the domain but nonhomogeneous Dirichlet
boundary condition, $u= y(1-y)$, on the left edge. The random
coefficient varies in the vertical direction and has the following
form
\begin{equation}
a(\omega, \bx) = a_0 + \sigma \sqrt{\lambda_0} Y_0(\omega) +
\sum_{i=1}^{n_f} \sigma \sqrt{\lambda_i}\left[Y_i(\omega)\cos(i\pi
y) + Y_{n_f+i}(\omega)\sin(i\pi y)\right] \label{eq:a}
\end{equation}
with $\lambda_0 = \frac{\sqrt{\pi L_c}}{2}$, $\lambda_i =
\sqrt{\pi} L_c e^{-\frac{(i \pi L_c)^2}{4}}$ for $i=1, \ldots,
n_f$ and $Y_0$, \ldots, $Y_{2n_f}$ are uncorrelated random
variables with zero mean and unit variance. In the following numerical test,
we take $a_0= 1$, $L_c = 0.25$, $\sigma = 0.15$, $n_f = 3$ and
assume the random variables $Y_0, \ldots, Y_{2n_f}$ are
independent and uniformly distributed in the interval $[-\sqrt{3},
\sqrt{3}]$.
We use linear finite elements for spatial discretization and simulate the system over the
time interval $[0, 0.5]$.
This choice of final time guarantees a steady-state can be achieved at the end of simulations.
A similar computational setting is used in \cite{nobile2009analysis} to compare several numerical methods
for parabolic equations with random coefficients. 
{In the following tests, the uniform triangulation with the maximum mesh size $h=\sqrt{2}/32$ and the uniform time partition with the time step size $\Delta t = 2.5\times 10^{-3}$ are used. 
}

{   For the implementation of the EMC method as discussed in Section \ref{sec:random}, we first select a set of $J$ random samples by the MC sampling, then run our deterministic code for simulating the ensemble of the deterministic PDEs associated with the $J$ realizations. Since the numerical accuracy with respect to the mesh size and time step size for the deterministic case has been verified in the first example, here we only check the rate of convergence in the EMC approximation error with respect to the number of samples, $J$. As the exact solution is unknown, we choose the EMC solution using $J_0=5000$ samples as our benchmark, vary the values of $J$ in the EMC simulations, and then evaluate the approximation errors based on the benchmark. 
Furthermore, we repeat such error analysis for $M= 10$ independent replicas and compute the average of the output errors. 
Denote the EMC solution at time $t_n$ in the $m$-th independent replica by $\Psi_{J, h}^{n, m}= \frac{1}{J} \sum_{j=1}^J u_{j, h}^{n,m}$, where $u_{j, h}^{n,m}$ when $J$ sample points are considered. 
Define 
$$
\mathcal{E}_{L^2}= \max_{n\in \{1, \ldots, N\}}\sqrt{ \frac{1}{M} \sum_{m=1}^{M} \| \Psi_{J_0, h}^{n, m} - \Psi_{J, h}^{n, m} \|^2 }\,,
$$
$$
\mathcal{E}_{H^1}= \sqrt{ \frac{\Delta t}{M} \sum_{m=1}^{M}  \sum_{n=1}^N \| \nabla \Psi_{J_0, h}^{n, m} - \nabla \Psi_{J, h}^{n, m} \|^2 }\,.
$$
The numerical results at $J=10, 20, 40, 80, 160$ are listed in Table \ref{tab:rand_err}. 
We further apply the linear regression analysis on these numerical results, which shows $\mathcal{E}_{L^2}\approx 0.0032\,J^{-0.5133 }$ and $\mathcal{E}_{H^1}\approx 0.0104\, J^{-0.4877}$. 
The values of $\mathcal{E}_{L^2}$ and $\mathcal{E}_{H^1}$ together with their linear regression models are plotted in Figure \ref{fig:rand_err} respectively. 
It is seen that the rate of convergence with respect to $J$ is close to $-0.5$, which coincides with our theoretical results in Theorem \ref{th:rand_error}.
}

\begin{figure}[htp]
\centering
\includegraphics[width=.5\textwidth]{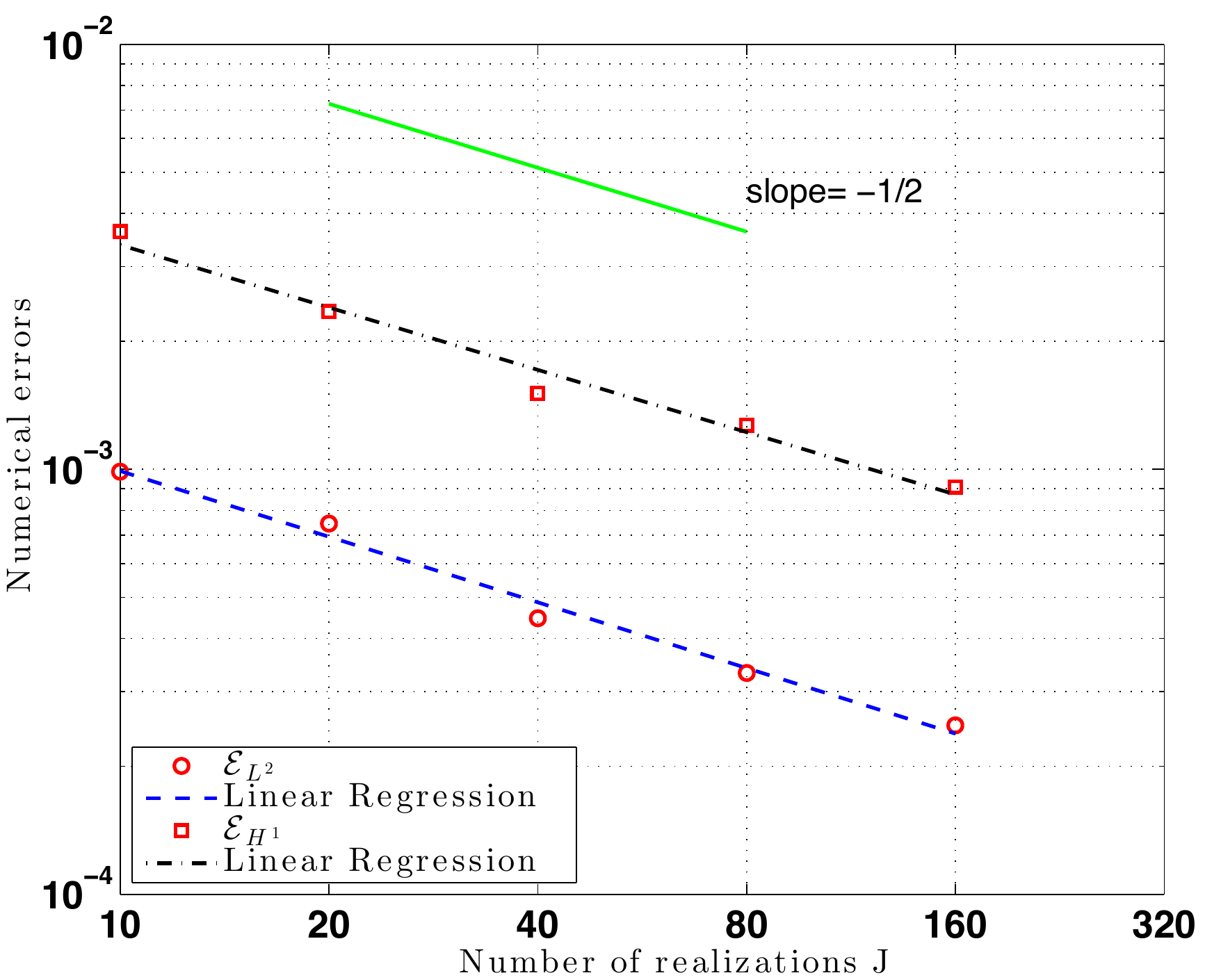}
\caption{Ensemble simulations: errors converge on the order of $\mathcal{O}(1/\sqrt{J})$.} 
\label{fig:rand_err}
\end{figure}

\begin{table}[htp]
\centering
{\footnotesize
\caption{Numerical errors of ensemble simulations.} \label{tab:rand_err}
\begin{tabular}{|c|c|c|c|c|c|}
\hline
$J$& 10 & 20 & 40 & 80 & 160 \\
\hline
$\mathcal{E}_{L^2}$ & 9.8672e-04 & 7.4471e-04 & 4.4640e-04 & 3.3151e-04 & 2.4966e-04 \\
$\mathcal{E}_{H^1}$ & 3.6240e-03 & 2.3475e-03 & 1.5080e-03 & 1.2696e-03 & 9.0903e-04 \\
\hline
\end{tabular}
}
\end{table}
Next, we analyze some stochastic information of the system {   including the expectation of $u$ at the final time and a quantity of interest.
In particular,  we are mainly interested in comparing the ensemble simulation outputs with those from the individual simulations when the same set of samples is used.} 
More precisely, we approximate the expected value 
$E[u(\omega, \bx, T)]$ by the EMC approximation  
 $$\Psi_h^E(\bx) \coloneqq \frac{1}{J}\sum_{j=1}^J u_{h}^{E}(\omega_j, \bx, T),$$
{    where $u_{h}^{E}(\omega_j, \bx, T)$ is the $j$-th member solution in the ensemble-based simulation at time $T$. 
Taking} the number of sample points to be $J=5000$, we compute the mean and standard deviation of the solutions at the final time, which are plotted in Figure \ref{fig:rand_mean_std} (left and middle).
\begin{figure}[htp]
\centering
\begin{minipage}{0.32\textwidth}
\includegraphics[width=.9\textwidth]{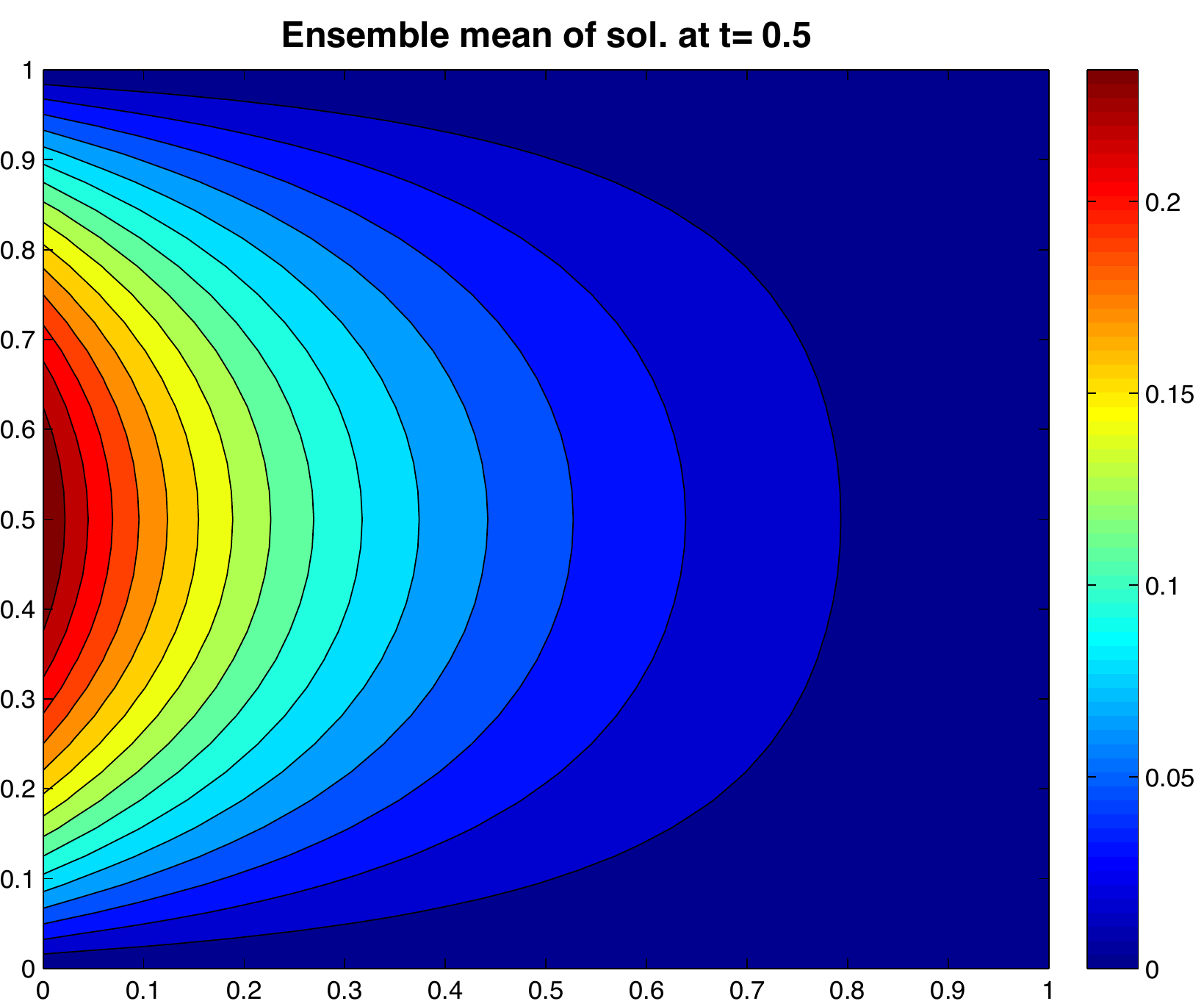}
\end{minipage}
\begin{minipage}{0.32\textwidth}
\includegraphics[width=.9\textwidth]{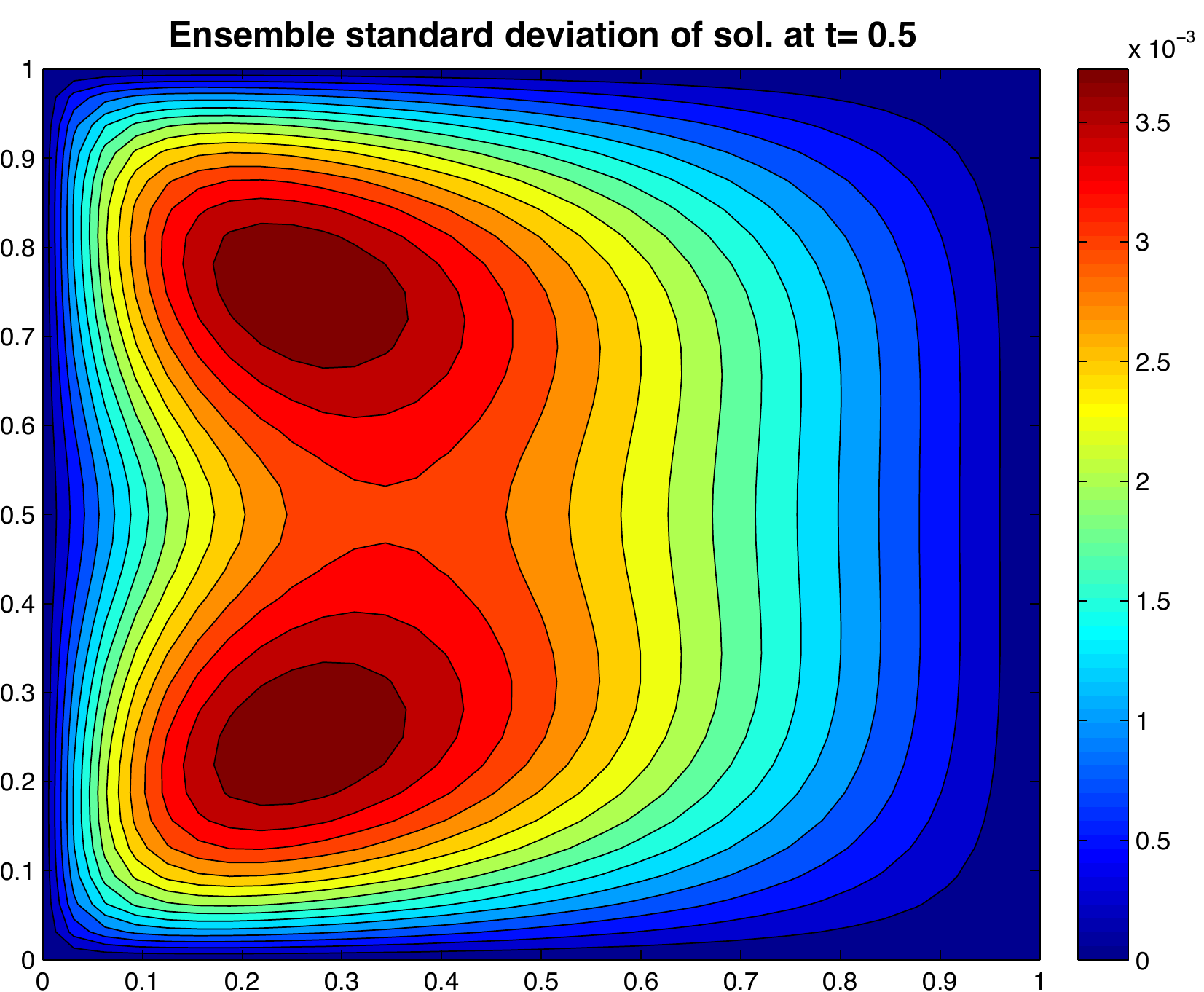}
\end{minipage}
\begin{minipage}{0.32\textwidth}
\includegraphics[width=.9\textwidth]{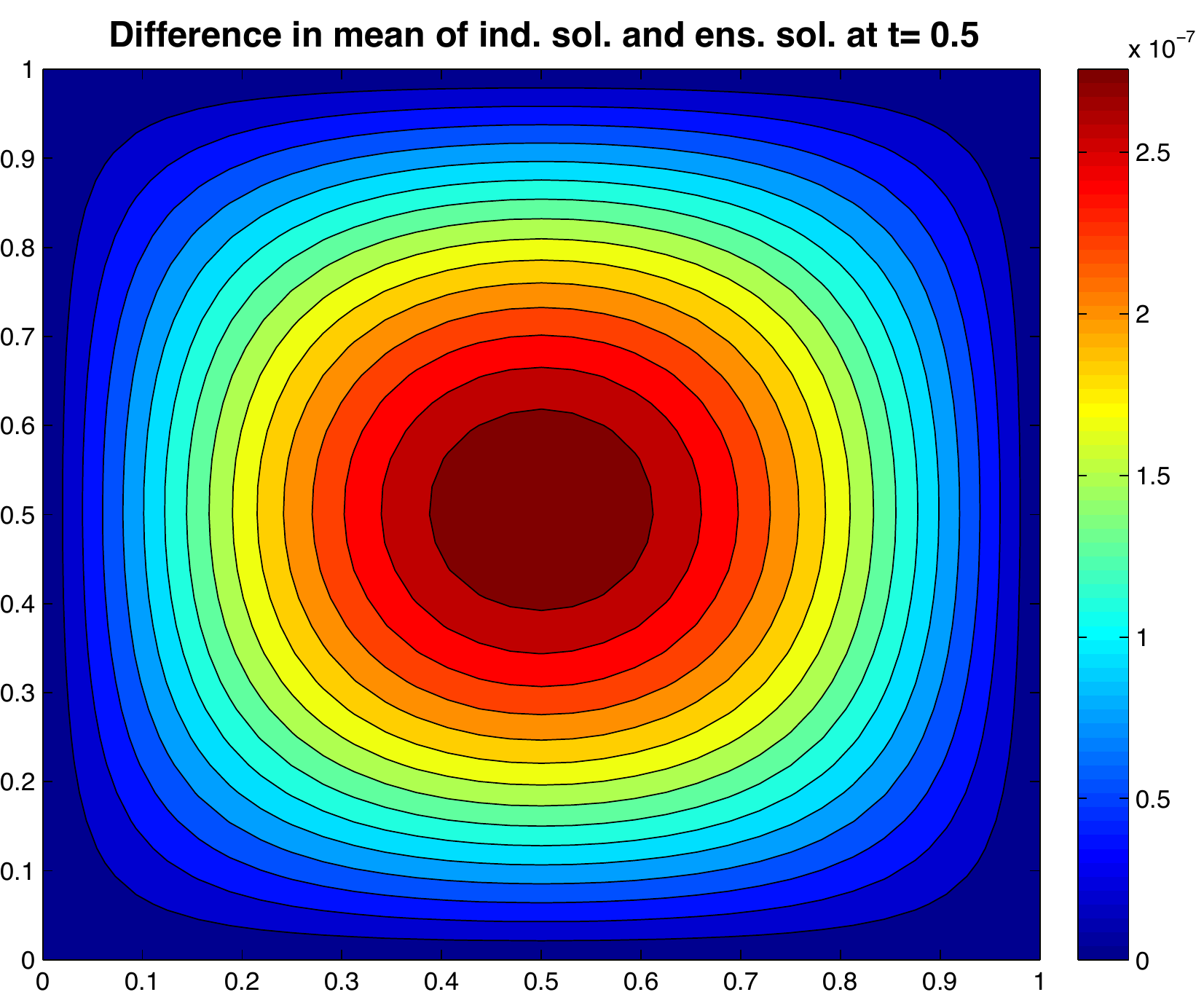}
\end{minipage}
\caption{Ensemble simulations: mean (left), standard deviation
(middle) of solution at $t= 0.5$, difference of the mean from that of the individual FEMC simulations (right).} \label{fig:rand_mean_std}
\end{figure}
To quantify the performance of the EMC method, we compare the result with that of individual finite element Monte Carlo (FEMC)  simulations using the same set of sample values.
Denote the FEMC approximation solution by
 $$\Psi_h^I(\bx) \coloneqq \frac{1}{J}\sum_{j=1}^J u_{h}^{I}(\omega_j, \bx, T),$$
{   where $u_{h}^{I}(\omega_j, \bx, T)$ is the $j$-th independent solution at time $T$. }
The difference between $\Psi_h^E$ and $\Psi_h^I$ is shown in Figure \ref{fig:rand_mean_std} (right).
It is observed that the difference is on the order of $10^{-7}$, which indicates the EMC method is able to provide the same accurate approximation as the individual FEMC simulations.
{However, the CPU time for the ensemble simulation is $1.4494\times 10^4$ seconds, while that of the individual simulations is $4.9197\times 10^4$ seconds. The former improves the computational efficiency by about 70\%.
Here, because the size of discrete system is small, we use Matlab and its LU matrix factorization for solving the linear system.}

Following \cite{nobile2009analysis}, we also calculate the quantity of interest
$$Q(\omega) = \int_D u(\omega, \bx, T)\,d D.$$
When the sample size is $J= 5000$, the histogram of the quantity of interest obtained from
the ensemble simulations, $Q_h^{E}$, is shown in Figure \ref{fig:rand_hist}
(left) with a fitted gaussian distribution.
We then do a comparison with the quantity of interest, $Q_h^{I}$, achieved from individual simulations at the same sampling set.
The histogram of the differences in absolute value, $\left|Q_h^{E}-Q_h^{I}\right|$, is plotted in
Figure \ref{fig:rand_hist} (right), which also illustrates that the EMC method outputs a close quantity of interest to the
standard FEMC method.

\begin{figure}[htp]
\begin{minipage}{0.49\textwidth}
\includegraphics[width=1\textwidth]{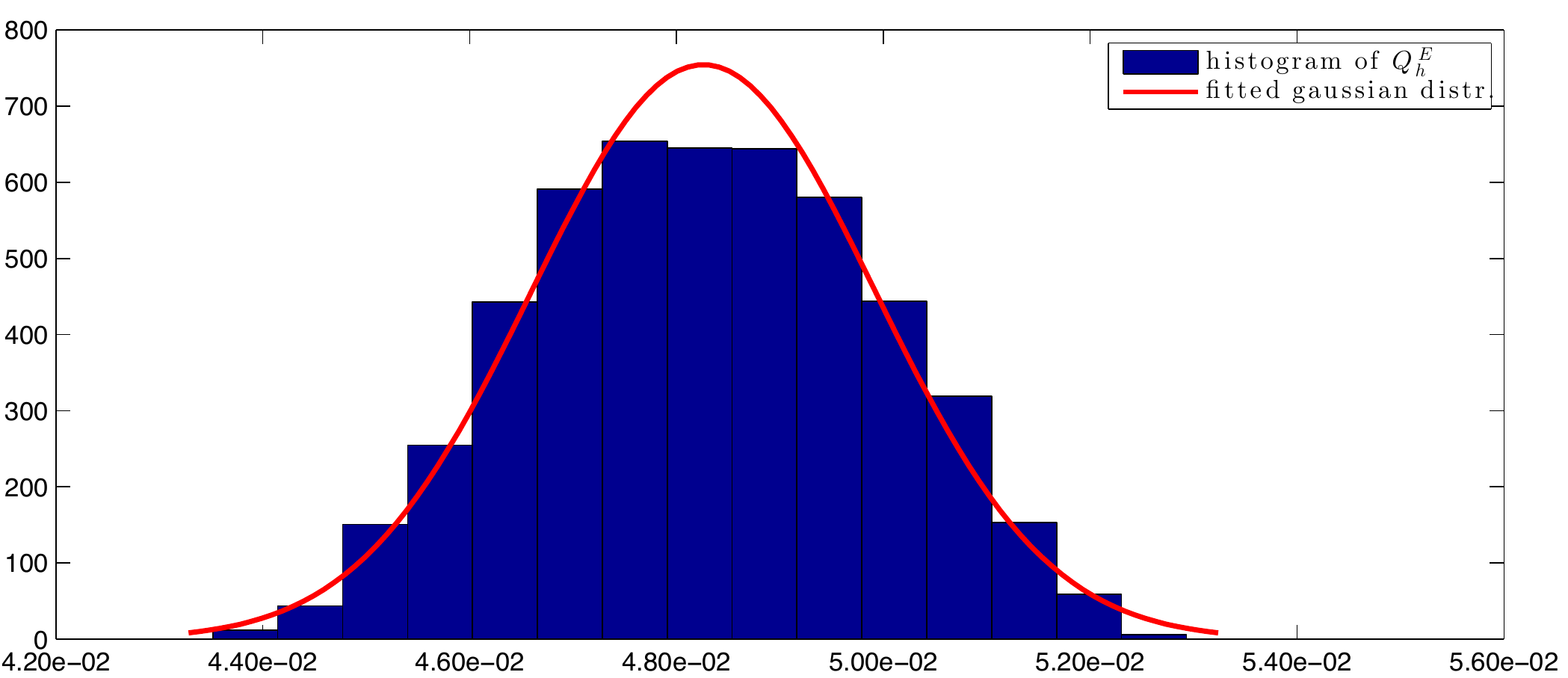}
\end{minipage}
\begin{minipage}{0.49\textwidth}
\vspace{.8em}
\includegraphics[width=1\textwidth]{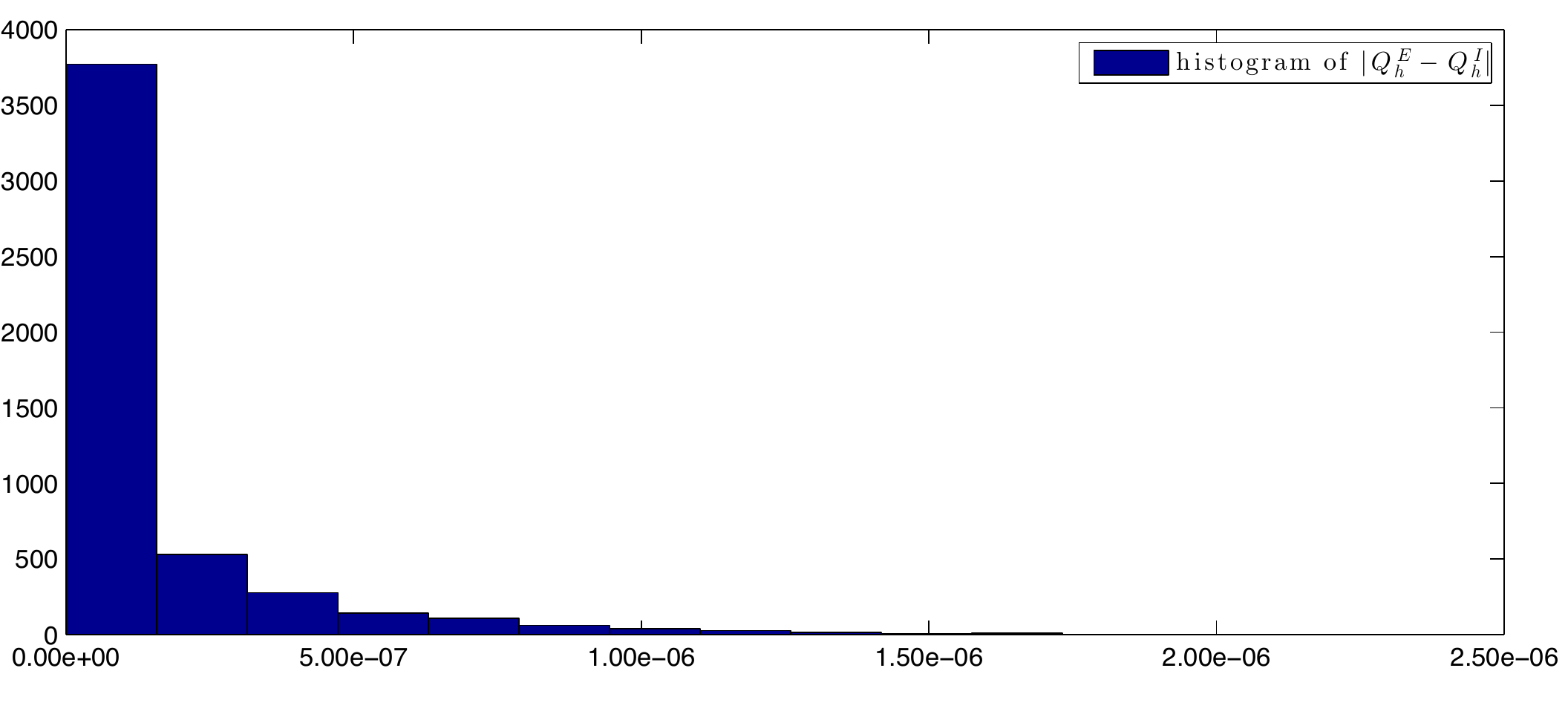}
\end{minipage}
\caption{Ensemble simulations: histogram of the quantity of
interest $Q_h^E$ with a fitted gaussian distribution (left); histogram of
the difference of the quantity in absolute value, $\left|Q_h^{E}-Q_h^{I}\right|$, between ensemble
simulations and individual simulations (right).}
\label{fig:rand_hist}
\end{figure}


\section{Conclusion}
\label{sec:con}
We propose an ensemble-based algorithm in this paper to improve the computational efficiency for a group of numerical solutions to parabolic problems.
The fundamental idea is to turn the linear systems associated to the group {into} a linear system with multiple right-hand-side vectors, which would reduce the computational time.
We first analyze the ensemble scheme for deterministic equations, then develop the ensemble-based Monte Carlo method for stochastic equations.
The effectiveness of both cases is demonstrated through rigorous error estimates and {illustrated with numerical experiments}.
The approach can be easily extended to more general, nonlinear parabolic equations, which is one of the research directions we are pursuing.

\section*{Acknowledgments}
The first author would like to thank the support of China Scholarship Council for visiting the Interdisciplinary Mathematics Institute at University of South Carolina during the year 2016--2017. {  We also thank the anonymous referees for their comments and suggestions, which significantly improved this paper.}

\newpage
\bibliographystyle{siamplain}
\bibliography{Ensemble}
\end{document}